\documentclass[reqno,12pt,a4paper]{amsart}
\usepackage{amsmath,amssymb,amsthm,amscd,amssymb,amsfonts,hyperref}
\usepackage[english]{babel}
\usepackage[mathscr]{eucal}

\numberwithin{equation}{section}

\renewcommand{\a}{\alpha}
\renewcommand{\b}{\beta}
\newcommand{\g}{\gamma}

\renewcommand{\d}{\delta}

\newcommand{\e}{\epsilon}

\newcommand{\z}{\zeta}

\renewcommand{\l}{\lambda}

\newcommand{\m}{\mu}
\newcommand{\n}{\nu}

\newcommand{\h}{\chi}

\renewcommand{\o}{\omega}
\renewcommand{\O}{\Omega}

\newcommand{\C}{{\mathbb C}}
\newcommand{\R}{{\mathbb R}}
\newcommand{\Z}{{\mathbb Z}}

\newcommand{\Fc}{{\mathcal F}}

\DeclareMathOperator{\rank}{rank}

\newtheorem{theorem}{Theorem}[section]
\newtheorem{proposition}[theorem]{Proposition}
\newtheorem{lemma}[theorem]{Lemma}

\theoremstyle{definition}

\theoremstyle{remark}

\begin{document}

\title[Finite rank Toeplitz operators in many dimensions]{Finite rank Bergman-Toeplitz
and Bargmann-Toeplitz operators in many dimensions}
\author[Rozenblum]{Grigori Rozenblum}
\address[G. Rozenblum]{Department of Mathematics \\
                          Chalmers University of Technology \\
                          and Department of Mathematics University of Gothenburg \\
                          Chalmers Tv\"argatan, 3, S-412 96
                           Gothenburg
                          Sweden}
\email{grigori@math.chalmers.se}
\author[Shirokov]{Nikolai Shirokov}
\address[N. Shirokov]{Department of Mathematical Analysis
\\Faculty of Mathematics and Mechanics\\
                        St. Petersburg State University\\
                        2, Bibliotechnaya pl.,
                         Petrodvorets\\
St.Petersburg, 198504, Russia}
\email{matan@math.spbu.ru}
\begin{abstract}
The recent theorem by D. Luecking that finite rank
Toeplitz-Bergman operators must be generated by a
measure consisting of finitely many point masses is
carried over to the many-dimensional case.
\end{abstract}
\keywords{Bergman spaces, Bargmann spaces, Toeplitz
operators}
\date{}

\maketitle


\section{Introduction and the main result}
\label{intro}
 Toeplitz operators play
an important role in many branches of analysis. For
Toeplitz operators, as well as for Hankel operators of
some types, the following cut-off property is often
encountered: if the operator of a certain class is too
`small' then it must have a rather special form or
even be zero. `Smallness' is measured most often in
the terms of compactness or membership in Schatten
classes.  One of such cut-off questions, about
Toeplitz-Bergman and Toeplitz-Bargmann operators  has
been under discussion for rather long time. Suppose
that such operator, with a certain weight measure, has
finite rank. What can one say about the measure?

We identify $\R^{2d}$ and $\C^d$ in the standard way:
$$
\C\ni z_j=x_j+iy_j\sim (x_j,y_j)\in \R^2,\
j=1,2,\dots,d; \ z=(z_1,\dots,z_d)\in\C^d.
$$
Let $d m(z)$ be the Lebesgue measure in $\C^d$ and
$d\g(z)=(2\pi)^{-d}e^{-|z|^2/2}dm(z)$. In the space
$L^2=L^2(\C^d,\m)$, the closed subspace $\Fc^2$ of
entire analytic functions is considered. Denote by $P$
the orthogonal projection in $L^2$ onto $\Fc^2$. This
is an integral operator  with the kernel
\begin{equation}\label{1:Bkernel}P(z,w)=\exp(z\bar{w}/2).\end{equation}

Next, let $V$ be a bounded complex  valued function on
$\C^d$, referred to as the weight function in what
follows. The Toeplitz operator we are interested in
(we call it Toeplitz-Bargmann operator) is defined by
$$
     T_V^\Fc:\Fc^2\to\Fc^2, \; T_V:u\mapsto PVu.
$$
Equivalently, the operator $T_V$ in $\Fc^2$ is defined
in terms of its sesquilinear form:
\begin{equation}\label{1:BargForm}
(T_V^\Fc u,v)_{\Fc^2}=
\int_{\C}V(z)u(z)\overline{v(z)}d\g(z), \quad u,v\in
\Fc^2.
\end{equation}

The operator $T_V^\Fc$ can be also expressed as an
integral operator of the form
\begin{equation}\label{1:bargkernel}
    (T_V^\Fc u)(z)=\int_{\C^d}u(w)P(z,w)V(w)d\g(w),
\end{equation}
where $P(z,w)$ is the kernel \eqref{1:Bkernel}. This
latter expression opens the possibility to consider
Toeplitz operators with a measure playing the role of
the weight function. In $\m$ is a complex regular
Borel measure on $\C^d$, having a bounded support, we
can consider the operator $T_\m^\Fc$ defined by the
relation
\begin{equation}\label{1:BargMeas}
(T_\m^\Fc u)(z)=\int_{\C^d}u(w)P(z,w)d\mu(w).
\end{equation}
The operator \eqref{1:BargMeas} takes the form
\eqref{1:bargkernel} if the measure $\mu$ is
absolutely continuous with respect to $\g$, with $V$
being the derivative $\frac{d\mu}{d\g}$.

Along with Toeplitz-Bargmann operators we consider
Toeplitz-Berg/-man operators. Let $\O$ be a bounded
connected domain in $\C^d$ and $B(z,w)$ be its Bergman
kernel, the integral kernel of the projection $P_{\O}$
of the space $L^2(\O)$ onto the Bergman space
$H^2(\O)$ consisting of holomorphic functions in
$L^2(\O)$. For a complex regular Borel measure $\mu$
on $\O$, having compact support in $\O$, we consider
the Toeplitz-Bergman operator $T^\O_\m$
\begin{equation}\label{1:BergT}
    (T^\O_\m u)(z)=\int B(z,w)u(z)d\mu(z), \ u\in H^2(\O).
\end{equation}
The operator can be described by the quadratic form
\begin{equation}\label{1:BergQform}
    (T^\O_\m u,v)=\int u(z)\overline{v(z)}d\mu(z)
\end{equation}
Again, if the weight measure $\mu$ is absolutely
continuous with respect to Lebesgue measure with
derivative $V$, the operator, denoted now by $T^\O_V$,
can be represented in the standard Toeplitz form
\begin{equation}\label{1:BergTProj}
T^\O_V u=PVu.
\end{equation}

It is clear that if the measure $\m$ is a linear
combination of finitely many point masses,
$\m=\sum_{k=1}^N \l_k \d(z-z_k)$, $\l_k\ne0$ then
 the operators
$T_\m^\Fc, \ T_\m^\O$ have finite rank, with the range
coinciding with the linear span of functions
$P(z,z_k)$, resp., $B(z,z_k)$,

\begin{equation}\label{rank}
    T_\m^\O u(z)=\sum B(z,w_k)\l_k u(w_k).
\end{equation}
 So, the natural question arises wether the converse
is true: if the operators have finite rank does this
imply that the measure consists only of a finite set
of point masses. For  absolutely continuous measures
this would mean that finite rank Toeplitz operators
must be zero.

The result in the case  of rank zero is folklore. If
the operator  (Bargmann or Bergman) is zero than the
measure should be zero, and it follows easily from
Stone-Weierstrass theorem. In 1987 in \cite{Lue} a
proof was proposed of the finite rank conjecture for
$d=1$, but it was seriously flowed. In 2002 an attempt
to prove this conjecture, again for $d=1$, was made by
N. Das in \cite{Das},  however, again, with
incorrigible mistakes, see the review \cite{Englis}.

The authors became interested in the finite rank
conjecture due to its relation to the study of the
spectral properties of the  perturbed Landau
Hamiltonian. The unperturbed Landau operator describes
the movement of a charged quantum particle confined to
a plane under the action of a uniform magnetic field.
This operator has spectrum consisting of Landau
levels, infinitely degenerated eigenvalues placed at
the points of an arithmetic progression. The
corresponding eigenspaces are explicitly expressed via
the Bargmann space. When the Landau operator is
perturbed by a compactly supported (or fast decaying)
electrostatic potential or magnetic field the Landau
levels split into clusters of eigenvalues, having
Landau levels as their only limit points. The
distribution of perturbed eigenvalues in  clusters is
essentially governed by the spectrum of
Toeplitz-Bargmann operators with weight function $V$
expressed in an explicit way in the terms of the
perturbation. Many results in this direction have been
obtained in \cite{RaiWar}, \cite{MelRoz},
\cite{RozTa}, \cite{RozSob} and other publications. In
particular, simple operator-theoretical arguments,(see
e.g. \cite[Proposition 4.1]{RaiWar}) show that the
Landau level is, in fact, the accumulation point of a
cluster if and only if $T_V^B$ has infinite rank. So,
if $T_V^B$ has finite rank, the Landau level remains,
even after the perturbation, being an isolated
eigenvalue of infinite multiplicity. The affirmative
answer to the finite rank conjecture would mean that
under a non-zero perturbation the Landau levels
necessarily split into infinite clusters.

The authors, together with A. Pushnitsky, spent some
time in 2005, trying to prove the conjecture. Certain
partial results were obtained, and a text was in
preparation, when a beautiful proof by Dan Luecking
\cite{Lue2} appeared. In that paper, in  the case
$d=1$, the finite rank conjecture was proved for the
operators $T_\mu^B$ and $T_\mu^\O$ without any extra
conditions. Being quite impressed, we stopped our
work.

A year later, the authors decided to return to the
problem, in its multi-dimensional setting. Besides the
natural curiosity, we were moved by fact  that the
proof in \cite{Lue2} used essentially the specifics of
the case $d=1$, while the higher-dimensional case is
also of interest for applications, see, e.g.,
\cite{MelRoz}. As a result we managed to carry over
the result of \cite{Lue2} to the multi-dimensional
case in its  full generality.

\begin{theorem}\label{Theorem} Suppose that  for a
certain finite measure $\mu$, compactly supported in
$\C^d$ or in $\O$, the corresponding operator
$T_\mu^\Fc$, resp., $T_\mu^\O$ has finite rank $N$.
Then the measure $\mu$ is a linear combination of $N$
point masses. In particular, if the measure $\m$ is
absolutely continuous with respect to Lebesgue measure
them it is zero.
\end{theorem}

We should note that the condition of the measure to
have compact support cannot be dropped altogether. In
\cite{Vas} an example of a non-trivial measure $\m$
without compact support was constructed such that the
corresponding operator $T_\mu^\Fc$ is zero.

The authors express their thanks to Dan Luecking who
provided us with the copy of his new paper on a very
early stage  and to Alex Pushnitsky for co-operation
during the work on the ill-fated original text. The
second author was partially supported by a grant from
the Swedish Academy of Sciences. The results were
obtained while he was enjoying the hospitality of
Chalmers University of Technology, Gothenburg.

\section{General properties}\label{2:soft}
The first simple observation, used in a slightly
different form in \cite{Lue2}, shows that the finite
rank property implies the finite rank for a certain
infinite matrix. For  the given measure $\m$ we
consider the infinite matrix
\begin{equation}\label{2:matrix}A(\m)=(a_{\a\b}): a_{\a\b}=\int z^\a \bar{z}^\b
d\mu(z),\end{equation}
 where
$\a,\b$ are multi-indices in $\Z_+^d$.
\begin{lemma}\label{PolyLemma} Let the operator $T=T_\m^\Fc$ or $T=T_\m^\O$
have finite rank. Then the matrix $A(\m)$ has finite
rank, $\rank(A(\m))\le \rank(T)$\end{lemma}

In fact, if we had $\rank (A(\m))>N=\rank(T)$, this
would mean that for some $M$ the dimension of the
range of the operator $T$ restricted to the
finite-dimensional subspace consisting of polynomials
of degree less than $M$ is greater than the dimension
of the range of $T$ itself, which is impossible.

At the moment we cannot establish the converse to
Lemma \ref{PolyLemma} in full generality. This result
will follow from Theorem \ref{Theorem}, as it will be
explained later. However for some nice domains $\O$ we
can prove the converse right now.

\begin{lemma}\label{2:lemma.polydisk} Let the domain
$\O$ be a polidisk in $\C^d$. Then $\rank(A(\m))=\rank
(T_\mu^\O)$.\end{lemma}
\begin{proof} We must show that if
$\rank(A(\m))=M<\infty$ then $\rank(T)\le M$,
$T=(T_\mu^\O))$. Suppose that $\rank(T)> M$. Then
there must exist $M+1$ functions $u_1,\dots
,u_{M+1}\in H^2(\O)$ such that the $Tu_j$ are linearly
independent. So the matrix $\tilde{A}(\mu)$ with
entries
\begin{equation}\label{2:matrixTilde}
    \tilde{a}_{j,k}=(Tu_j,u_k)=\int
    u_j(z)\overline{u_k(z)}d\m(z),\ 1\le j,k\le M+1,
\end{equation}
is nonsingular. However in a polidisk any holomorphic
function can be uniformly on compact sets approximated
by polynomials, the starting segments of the Taylor
series. If $p_j$ are polynomials approximating $u_j$,
the matrix with elements $(Tp_j,p_k)$ has rank not
greater than $M$ and therefore singular, so it cannot
approximate the nonsingular matrix
$\tilde{A}(\mu)$.\end{proof}
 The same reasoning, using this approximation property and density
  of polynomials in the Bargmann space  proves the similar statement for
 the Bargmann operators.
 \begin{lemma}\label{lemmaBarg}
 For a measure $\mu$ with compact support $\rank(A(\m))=\rank
(T_\mu^\Fc)$.\end{lemma}

 Lemma \ref{2:lemma.polydisk}
enables us to reduce the Bargmann case to the Bergman
one. Suppose that for some measure $\mu$ the operator
$T_\m^\Fc$ has finite rank. Then, by Lemma
\ref{PolyLemma}, the matrix $A(\m)$ has finite rank.
Consider a polidisk $\O$ containing the support of
$\m$ in a compact set inside. Since the matrix $A(\m)$
is the same for operators $T_\m^\Fc$ and $T_\m^\O$, it
follows from Lemma \ref{2:lemma.polydisk} that
$T_\m^\O$ has finite rank.

So it is sufficient to prove Theorem \ref{Theorem} for
the Bergman case only. We will suppress the sperscript
$\O$ further  on.

Now we introduce a functional parameter. Let $g(z)$ be
a bounded holomorphic function in $\O$. We consider
the measure $\m_g=|g|^2\m$.

\begin{lemma} Suppose that the operator $T_\m$ has
finite rank. Than the operator $T_{\m_g}$ has finite
rank as well and $\rank(T_{\m_g})\le \rank(T_{\m})$.
If the function $g(z)^{-1}$ is also bounded and
holomorphic than $\rank(T_{\m_g}^\O)=
\rank(T_{\m}^\O)$.\end{lemma}
\begin{proof} By considering  the quadratic forms,  we see that operator $u\mapsto gT_{\m}^\O
gu$ coincides with $T_{\m_g}^\O$, the multiplication
by a bounded analytical function, being a bounded
operator in the Bergman space, cannot increase the
rank, and therefore the first statement follows. The
second statement is now obvious.\end{proof}

\section{Proof of the main result}

We are going to prove Theorem \ref{Theorem} using
induction in dimension. It will follow from a more
general result.

\begin{proposition}\label{Propos} Let $\mu$ be a
regular Borel measure  with compact support in $C^d$.
Suppose that for any  function $g(z)$, bounded and
holomorphic in a fixed polidisk neighborhood of the
support of $\mu$ such that $g(z)^{-1}$ is also bounded
and holomorphic, the matrix $A(\mu_g)$ has rank $N$.
Then the measure $\mu$ is a linear combination of  $N$
point masses.\end{proposition}

\begin{proof} The base of induction is the result by
D. Luecking in \cite{Lue2} for $ d=1$. In fact, it is
even sufficient here to have the finite rank of the
matrix $A(\mu)$ only, without the extra function $g$.

Now we perform the step of induction. Suppose that the
statement is proved in dimension $d-1$. Consider the
$d$-dimensional case.  We denote the co-ordinates in
$\C^d$ by $z=(z_1,z')$, $z'\in\C^{d-1}$.
 For a bounded holomorphic function
$g(z)$, $z\in D$, we denote $\m_g=|g|^2\m$.

Let $\pi$ denote the projection, $\pi(z_1,z')=z'$. For
the measure $\mu$ on $\C^d$ we will denote by
$\n=\pi_*\mu$ the induced measure on $\C^{d-1}$,
\begin{equation}\label{InducedMeasure}
    \n(E)=\m(\pi^{-1}E),
\end{equation}
for Borel sets $E\subset\C^{d-1}$. If $\mu$ is a
regular Borel measure, the same is true for $\n$.

In the matrix $A(\mu_g)$ we consider the sub-matrix
$A'(\mu_g)$ consisting of entries $a_{\a\b}$ with
multi-indices $\a,\b$ having zero as their first
component, $\a_1,\b_1=0$. Then the element $a_{\a\b}$
can be written as
\begin{equation}\label{submatrix}
a'_{\a'\b'}=a_{(0,\a')(0,\b')}=\int{z'}^{\a'}{\overline{z'}}^{\b'}\mu_g(z)=\int{z'}^{\a'}{\overline{z'}}^{\b'}\nu_g(z'),
\end{equation}
the last equality expressing the fact that the
integrand is independent of $z_1$.
 Since the matrix $A'(\mu_g)$ is a submatrix of the
finite rank matrix $A(\mu_g)$, its rank is not greater
than the rank of $A(\mu_g)$, and therefore
\begin{equation}\label{matrixrank}
\rank A'(\mu_g)\le N,
\end{equation}
for any function $g$ satisfying the conditions above.
This means that the  measure $\nu_{g}$ satisfies the
conditions of the inductive hypothesis in dimension
$d-1$, and therefore the measure $\nu_{g}$ is a finite
combination of point masses,
\begin{equation}\label{finite1}
    \nu_{g}=\sum_{j=1}^M \l_j\d(z'-\z_j),
\end{equation}
where $M\le N$, coefficients $\l_j\ne0$ and the points
$\z_j\in \C^{d-1}$ may depend, generally, on the
function $g$: $M=M(g)$, $\l_j=\l_j(g), \
\z_j=\z_j(g)$.

We are going to show now that, at least locally, the
points $\z_j$ do not depend on $g$.

The number $M(g)$, considered as a function of the set
of functions $g$, attains its maximal value,
$M_0=\max_g(M(g))\le N$ at some function $g=g_0$.
Without loss in generality, we can assume that already
$g_0=1$ gives the extremal value of $M(g)$ (otherwise,
we re-denote $\mu_{g_0}$ by $\m$). For any $j$, $1\le
j\le M_0$, we consider the measure $\mu^j$ in $\C^d$
defined as $\mu^j(G)=\mu(G\cap\pi^{-1}(\{\z_j(1)\}))$,
$G\subset\C^d$. So the measure $\mu^j$ is supported in
the pre-image of the point $\z_j(1)$ under the
projection $\pi$. Moreover, the coefficients $\l_j(1)$
in front of point masses $\d(z'-\z_j(1))$ are equal to
$\nu(\{\z_j(1)\})=\mu^j(\pi^{-1}\{\z_j(1)\})=\int 1
d\mu^j(z)$. Now, for a function $g$ as above, we can
express $\nu_g(\{\z_j(1)\})$ as
\begin{equation}\label{finite2}
\nu_g(\{\z_j(1)\})=\int |g(z)|^2 d\mu^j(z).
\end{equation}
By \eqref{finite2}, the quantities $\nu_g(\{\z_j(1)\}$
depend continuously on the function $g$ in the
topology of uniform convergence on the support of the
measure $\mu$. Thus, since they are not zero for
$g=1$, they are not zero for $g$ sufficiently close to
$1$ in the above topology, so, for such $g$, the point
masses at the point masses at $\z_j(1)$ are present in
the measure $\nu_g$ with nonzero coefficients. However
since there are $M_0$ points $\z_j(1)$, and $M_0$ is
the maximal possible quantity of the point masses in
$\nu_g$, this means that no more point masses appear.
Thus, for functions $g$ sufficiently close to $1$ in
uniform norm on the support of $\m$, the points
$\z_j(g)$ do not actually depend on $g$, in other
words, for such $g$ the measure $\n_g$ is a sum of
point masses placed at the points not depending on
$g$.

Now we consider the measure
$\widetilde{\m}=\m-\sum\m^j$, so $\widetilde{\m}(G)=0$
for any Borel set $G\subset\cup\pi^{-1}(\{\z_j(1)\})$.
Therefore $\pi_*\widetilde{\m}_g(\{\z_j(1)\})=0$ for
any $g$. At the same time, since for $g$ close to $1$
the support of $\nu(g)$ consists only of the points
$\z_j(1)$, we have $\pi_*\widetilde{\m}_g(E)=0$ for
any Borel set $E\subset\C^{d-1}$ not containing the
points $\z_j(1)$. Taken together, these two properties
mean that
\begin{equation}\label{zeroproj}
    \pi_*\widetilde{\m}_g=0
\end{equation}
for functions  $g$ sufficiently close to $1$. In
particular, we have
\begin{equation}\label{zeromeasure}
    \pi_*\widetilde{\m}_g(\C^{d-1})=\int|g(z)|^2d\widetilde{\m}(z)=0.
\end{equation}
Now we extend the equality \eqref{zeromeasure} from
functions $g$ which are close to $1$ to all functions
$g$, analytical in the polidisk neighborhood of the
support of $\m$. In fact, for a given function $g$ set
$g_\e=1+\e g$ and apply \eqref{zeromeasure} for
$g_\e$, with any $\e$ small enough. We obtain
\begin{equation}\label{zeomeasureeps}
\int(1+2\e\Re g(z)+\e^2|g(z)|^2)d\widetilde{\m}(z)=0.
\end{equation}
Due to the arbitrariness of a small $\e$,
\eqref{zeomeasureeps} implies \eqref{zeromeasure}.

Now we can show that the measure $\widetilde{\m}$ is,
in fact, zero. Consider the algebra generated by
functions $|g|^2$. This algebra, obviously, satisfies
all conditions of Stone-Weierstrass theorem, therefore
any function continuous on a compact set can be
uniformly on this compact approximated by linear
combinations of the functions of the form $|g|^2$.
Therefore, the relation  \eqref{zeromeasure} implies
that $\int f(z)d\widetilde{\m}(z)=0$ for any
continuous function, and thus
\begin{equation}\label{measure is0}
\widetilde{\m}=0, \ \mu=
\sum\m^j=\sum\mu(G\cap\pi^{-1}(\{\z_j(1)\})).
\end{equation}
So, the support of the measure $\mu$ is contained in
no more than $N$ complex planes $z'=\z_j(1),
z_1\in\C^1$, $j=1,\dots,M_0$.

Now we can  repeat the same reasoning but considering
the splitting of co-ordinates $z=(z'',z_d)$ and the
corresponding projection $z\mapsto z''$. We obtain
that the support of the measure $\mu$ is contained in
no more than $N$ complex planes $z''=\xi_k, z_d\in
\C^1, k\le N$. Surely, the support of $\mu$ must lie
in the intersection of these planes, which gives us no
more than $N^2$ points $Q(j,k):z'=\z_j,z''=\xi_k$.
Finally, to reduce the quantity of points in the
support of the measure, we rotate the coordinates in
$C^d$ by means of a unitary matrix to that in new
co-ordinates $\o=(\o_1,\dots,\o_d=(\o_1,\o'))$ the
points $Q(j,k)$ all have different $\o'$ co-ordinates.
Repeating our reasoning for the third time, we obtain
that the points $Q(j,k)$ lie on no more than $N$
complex planes $\o'=\h_l$, and since each of these
planes contains no more than one of the points
$H(j,k)$, this means that actually no more than $N$
points belong to the support of the measure. It is
clear that the number of point masses is exactly $N$.
If there were fewer of them, then the rank of the
Toeplitz operator would be smaller than
$N$.\end{proof}

Finally we can establish the converse to Lemma
\ref{PolyLemma} for an arbitrary domain $\O$. It shows
that the rank of the operator does not decrease if we
restrict it to polynomials, even if polynomials are
not dense in the space of holomorphic functions.

\begin{proposition}\label{2:lemma.arb} Let $\O$ be an arbitrary  domain
 in $\C^d$. Then $\rank(A(\m))=\rank
(T_\mu^\O)$.\end{proposition}
 \begin{proof}
By Theorem \ref{Theorem}, the measure $\mu$ consists
of $N$ point masses, $N=\rank(A(\m))$, with nonzero
coefficients. But now the relation \eqref{rank} shows
that the operator $T_\mu^\O$ has the same rank.
 \end{proof}

\end{document}